\documentclass[12pt]{amsart}
\usepackage{amsfonts, amsmath, amsthm}
\usepackage{graphicx}
\usepackage[left=1.5cm,top=1.5cm,right=2cm,includehead,includefoot]{geometry}
\usepackage{fourier}
\usepackage{microtype}

\newtheorem{pro}{Proposition}[section]
\newtheorem{thm}[pro]{Theorem}
\newtheorem{lem}[pro]{Lemma}

\newtheorem{cor}[pro]{Corollary}

\theoremstyle{definition}
\newtheorem{dfn}[pro]{Definition}

\newtheorem{notate}[pro]{Notation}

\theoremstyle{remark}
\newtheorem*{rmk}{Remark}

\newcommand{\wt}{\widetilde}

\newcommand{\del}{\partial}

\def\vol[#1]{\mbox{\rm Vol}(#1)}

\title[Tetrahedral number of manifolds of bounded volume]{A linear bound on
  the tetrahedral number of manifolds of bounded volume (after J{\o}rgensen and
  Thurston)}

\date{\today} 
\address{Department of Mathematics, Nara Women's University
Kitauoya Nishimachi, Nara 630-8506, Japan} \address{Department of mathematical
Sciences, University of Arkansas, Fayetteville, AR 72701}
\email{tsuyoshi@cc.nara-wu.ac.jp} \email{yoav@uark.edu} \author{Tsuyoshi
Kobayashi} \author{Yo'av Rieck} \thanks{The first names author was supported
by Grant-in-Aid for scientific research, JSPS grant number 00186751. The
second named author was supported in part by
the 21st century COE program ``Constitution for wide-angle mathematical basis
focused on knots" (Osaka City University); leader: Akio Kawauchi.\\
\copyright{2011 by the American Mathematical Society}
}

\begin{document}


\date{\today}%
\dedicatory{To Bus Jaco, on the event of his 70th birthday}%


\maketitle

\section{Introduction}
\label{sec:intro}

The purpose of this note is to prove Theorem~\ref{thm} below, due to J{\o}rgensen 
and Thurston and known to  experts in the field; the earliest reference we found 
to this theorem and a sketch of its proof appears in  Thurston's
notes~\cite[Chapter~5]{thurston}.
For  definitions and notation see the next section.  We study the triangulation of
the thick part.  For another geometric study of the triangulation of the thick part
see Breslin's \cite{breslin}.  We prove:

\begin{thm}[J\o rgensen, Thurston]
  \label{thm}
Let $\mu > 0$ be a Margulis constant for $\mathbb{H}^3$.  Then for any $d>0$ there exists
a constant $K  >0$, depending on $\mu$ and $d$,  so that for any complete finite volume 
hyperbolic 3-manifold $M$, $N_d(M_{[\mu,\infty)})$ 
can be triangulated using at most $K \vol[M]$ tetrahedra.
\end{thm}

\bigskip

\noindent  The manifold $N_d(M_{[\mu,\infty)})$ is the closed $d$-neighborhood of the 
$\mu$-thick part of $M$
and is denoted by $X$ throughout this paper.   By the Margulis Lemma, $M \setminus X$
consists of disjoint cusps and open solid tori, and each of these solid tori is a regular neighborhood of an embedded 
closed geodesic.  We refer to removing such an open solid torus neighborhood of a geodesic $\gamma$
as {\it drilling out} $\gamma$.  Thus $X$ is obtained from $M$ by drilling 
out short geodesics.  Note that any preimage of any component of $M \setminus X$
in the universal cover $\mathbb H^{3}$ is convex (for an explicit description see Definition~\ref{dfn:cone}).
In the next proposition we bring the basic facts about $X$;  this
proposition is independent of Theorem~\ref{thm}.  For a complete Riemannian manifold
$A$ and a point $a \in A$, we denote the radius of injectivity of $A$ at $a$ by
$\mbox{inj}_{A}(a)$. 

\begin{pro}
  \label{pro:X}
 Fix the notation of Theorem~\ref{thm}.    Then the following hold:
  \begin{enumerate}
  \item There exists $R = R(\mu,d) > 0$, independent of $M$, so that for any $x \in X$, 
  $\mbox{\rm inj}_{M}(x) > R$.
  \item $M_{[\mu,\infty)} \subset X \subset M$, and $X$ is obtained from $M$ by drilling out 
  short geodesics and truncating cusps.  In particular:
  \begin{enumerate}
  \item If $\gamma \subset M$ is a geodesic of length less than $2R$ then $\gamma$ is drilled out.
 \item If $\gamma \subset M$ is a simple geodesic of length at least $2\mu$ then $\gamma$ is not drilled out. 
  \end{enumerate}    
  \end{enumerate}  
\end{pro}

\begin{proof}[Proof of Proposition~\ref{pro:X}]
It follows from the decay of radius of injectivity (see, for example,  
\cite[Proposition~4.19]{decay}) that  there exists $R>0$, depending 
only on $\mu$ and $d$,  so that   for any $x \in X$, 
$\mbox{\rm inj}_{M}(x) > R$.  This establishes~(1).

By construction $M_{[\mu,\infty)} \subset X \subset M$.
Let $U$ be a component of $M_{(0,\mu]}$.  The set of points removed 
from $U$ is:
$$\{x \in U | d(x, M \setminus U) > d \}.$$ 
When $U$ is a solid torus
neighborhood of a closed short geodesic $\gamma$, the set of points removed 
is $\{x \in U | d(x,\gamma) \leq d(X,\gamma) - d\}$, and is 
either empty or an open solid torus neighborhood of $\gamma$.  In the
first case $U \subset X$ and in the second case we remove a neighborhood of 
$\gamma$.  When $U$ is a cusp, $(M \setminus X) \cap U$ is isotopic to $U$. 
This establishes~(2).  

Let $\gamma \subset M$ be a geodesic of length less than $2R$.  Then for every
$p \in \gamma$, $\mbox{inj}_{M}(p) < R$.  By~(1) above  
$\gamma$ is drilled out.  This establishes~(2)(a).

Let $\gamma \subset M$ be a simple geodesic of length at least $2\mu$.  
It is clear from the definitions that if $\gamma$
is drilled out then $\gamma \subset M_{[0,\mu)}$.  
A geodesic is contained in $M_{(0,\mu]}$ if and only if it covers a short geodesic (that is,
has the form $\delta^{n}$ for some geodesic $\delta$ with $l(\delta) < 2\mu$ and some $n>0$).  Such
a geodesic is simple if and only if $n=1$; we conclude that simple geodesics in $M_{(0,\mu]}$
are shorter than $2\mu$.  Thus $\gamma \not\subset M_{(0,\mu]}$, and it is not 
drilled out.  This establishes~2(b).
\end{proof}

\bigskip\bigskip

\noindent We denote by $t_C (M)$ the minimal number of tetrahedra 
required to triangulate a link exterior in $M$, that is, the minimal number
of tetrahedra required to triangulate $M \setminus \mathring N(L)$, where the minimum
is taken over all links $L \subset M$ (possibly, $L = \emptyset$) and all 
possible triangulations.  Similarly
we define $t_{HC}(M)$ to be the minimal number of tetrahedra necessary to triangulate
$M \setminus \mathring N(L)$, where $L \subset M$ ranges over all possible hyperbolic
links.  As a consequence of Theorem~\ref{thm} we get the following corollary,
showing that $\vol[M]$, $T_C(M)$, and $T_{HC}(M)$ are the same up-to
linear equivalence:

\begin{cor}
Let $\mu > 0$ be a Margulis constant for $\mathbb{H}^3$ and fix $d>0$.  Let $K>0$
be the constant given in Theorem~\ref{thm} and $v_3$ be the volume of a regular 
ideal tetrahedron in $\mathbb{H}^3$.  Then for any complete finite volume hyperbolic
3-manifold $M$ we have:
$$t_C(M) \leq t_{HC}(M) \leq K \vol[M] \leq K v_3 t_C(M) 
					\leq K v_3 t_{HC}(M).$$
\end{cor}

\begin{proof}
The first and last inequalities are obvious.

By Proposition~\ref{pro:X}~(2), $X=N_d(M_{\geq \mu})$ is obtained from $M$ by 
drilling out geodesics; hence by Kojima~\cite{kojima} it is a hyperbolic.
Thus the second inequality follows directly for Theorem~\ref{thm}.

The proof of third inequality is well known (see, for example, Chapter~C 
of~\cite{BenedettiPetronio}); we sketch its argument for the reader's convenience.
Let $L\subset M$ be a link and $\mathcal{T}$ a triangulation of
$M\setminus \mathring N(L)$ using $t_C(M)$ tetrahedra.
Let $\Delta$ be a 3-simplex in $\mathbb{R}^{3}$ and denote the
characteristic maps of the tetrahedra in $\mathcal{T}$ by $
\delta_{i}:\Delta \to M \setminus \mathring L$ ($i = 1,\dots,{t_C(M)}$).  
Let $f: M \setminus \mathring{N}(L) \to M$ be the degree 1 map obtained by crushing each
torus of $\del(M \setminus \mathring N(L))$ to a circle\footnote{That is, $f$ is obtained as 
following: foliate each component of $\partial M \setminus \mathring N(L)$ by circles, where each
leaf is isotopic to a meridian on $M$.  Then $f$ is obtained by identifying each leaf to a point.}.  
Denote $f  \circ \delta_i: \Delta \to M$ by $f_{i}$ ($i=1,\dots,t_{C(M)}$).  

With the notation as in the previous paragraph, we now prove the third inequality. 
Note that $\Sigma_{i=1}^{n} f_{i}$ represents a generator of $H_{3}(M) \cong \mathbb Z$.
Let $\tilde f_{i}$ be a lift of $f_{i}$ to the universal cover $\mathbb{H}^3$.
We construct a map $\bar f_{i}: \delta \to \mathbb H^{3}$ by ``pulling $\tilde f$ tight''\footnote{By 
``pulling $\tilde f$ tight'' we mean: for $p \in \Delta$ a vertex, then $\bar f(p) = \tilde f(p)$.
Next, for a general point $p \in \Delta$, $\bar f (p)$ is the unique point of $\mathbb H^{3}$ that has 
the same barycentric coordinates as $p$ (for more on barycentric coordinate see,
for example, \cite[Page 103]{hatcher}).}.  Note that $\bar f_{i}$ is homotopic to $\tilde f_{i}$; 
denote this homotopy by $\wt F_{i,t}(p)$.  Since $\Sigma_{i=1}^{n}f_{i}$ defines an
element of $H_3(M)$, faces must cancel in pairs.  Let $F$ and $F'$ be such a pair, 
that is, $f_{i}(F) = - f_{j}(F')$, and let 
$p \in F$ and $p' \in F'$ be corresponding points, that is, points with the same barycentric 
coordinates.   By preforming the homotopy at constant speed we obtain, for any $t \in [0,1]$:
$$\pi \circ \wt F_{i,t}(p) = \pi \circ \wt F_{j,t}(p').$$
Here $\pi$ is the universal cover projection.  This implies that for any $t \in [0,1]$, 
$\Sigma_{i=1}^{t_{C(M)}} \pi \circ \wt F_{i,t}(p)$  is homologous to  $\Sigma_{i=1}^{n} f_{i}$ 
and therefore represents a generator of $H_{3}(M)$.
For $t=1$, we see that $\Sigma_{i=1}^{t_{C(M)}} \pi \circ \bar f_{i}$ represents a 
generator for $H_{3}(M)$;
in particular, every point of $M$ is in the image of at least one $\pi \circ \bar f_{i}$.
Hence the sum of the volumes of the images of $\pi \circ \bar f_{i}$ is no less
than $\vol[M]$.  Using this, the fact that volumes do not increase under $\pi$, and the
fact that the volume of a hyperbolic tetrahedron is less than $v_{3}$ we get:
\begin{eqnarray*}
\vol[M] &\leq& \Sigma_{i=1}^{t_c(M)} \vol[\pi \circ \bar f_{i}] \\
             &\leq& \Sigma_{i=1}^{t_C(M)} \vol[\bar f_{i}]  \\
             &<& t_C(M) v_3.
\end{eqnarray*}
The third inequality follows.
\end{proof}

\bigskip

\noindent  The proposition below is the key to the proof of Theorem~\ref{thm} and is very 
useful in its own right.   For this proposition we need the following notation, that we will
use throughout this paper.  Fix a Margulis
constant $\mu > 0$ and $d>0$, and let $R>0$ be as in Proposition~\ref{pro:X}~(1).
We define $D = \min\{R,d\}$.
A set $A$ in a metric space is called {\it $D$-separated} if for any $p,q \in A$, 
$p \neq q$, we have that $d(p,q) > D$.
Fix $\{x_1\dots,x_N\} \subset X$ a generic set of $N$ points ({\it a-priori} $N$ may be infinite)
fulfilling the following conditions:
\begin{enumerate}
\item  $\{x_1\dots,x_N\} \subset X$  is $D$-separated.
\item  $\{x_1\dots,x_N\}$ is maximal (with respect to inclusion) subject to
  this constraint. 
\end{enumerate}
Let $V_1,\dots,V_N$ be the {\it Voronoi cells} in $M$ corresponding to  $\{x_1\dots,x_N\}$,
that is, 
$$V_i = \{p \in M |  d_M(p,x_i) \leq d_M(p,x_j) \ (j=1,\dots,N)\}.$$
We emphasize that {\it a-priori} a Voronoi cell need not be ``nice'';  
for example, it need not be a ball and may have infinite diameter.
Consider the following simple example: given any metric space and a single point in it,
the Voronoi cell corresponding to that point is the entire space.  

\begin{pro}
\label{pro:decomposition}
With the notation of Theorem~\ref{thm},there exists a constant $C = C(\mu,d)$ so that the
following holds:
	\begin{enumerate}
	\item $M$  is decomposed into $N \leq C \mbox{\rm Vol}(M)$ Voronoi cells. 
	\item $V_{i} \cap X$ is triangulated using at most $C$ tetrahedra ($i=1,\dots,N$)
	($V_{i} \cap X$ may not be connected).
	\item For any $i$, $i'$ ($i$, $i'=1,\dots,N$), the triangulations of $V_{i} \cap X$ and $V_{i'} \cap X$ given in~(2)
	above coincide on $(V_{i} \cap X) \cap (V_{i'} \cap X)$.
	\end{enumerate}
\end{pro}

We note that Theorem~\ref{thm} follows easily from Proposition~\ref{pro:decomposition}
by setting $K = C^{2}$.

\bigskip

\noindent {\bf Structure of this paper.}  In Section~\ref{sec:prelims} we cover some basic preliminary
notion.  In Section~\ref{sec:SetUp1} we describe the decomposition of $M$ into Voronoi
cells. In an attempt to make this paper
self-contained and accessible to all we provide proofs for many elementary 
facts about Voronoi cells.  In Section~\ref{sec:SetUp2} we study the intersection 
of the Voronoi cells with $X$.
In Section~\ref{sec:ProofProp} we prove Proposition~\ref{pro:decomposition}.  

\bigskip

\noindent{\bf Strategy.}  As mentioned, our approached is based on Thurston's 
original work.  However, as discussed in \cite{BenedettiPetronio} pp 190--192,
to make this work requires control over $V_{i} \cap X$.  We now briefly
explain our strategy for obtaining this control. 

We first decompose $M$ into the $N$ Voronoi cells described above.
An easy volume argument shows that $N$ is bounded above linearly in terms of
the volume of $M$.   We then show the following:
\begin{enumerate}
	\item Every component of $V_{i} \cap X$ is a handlebody, as we will show it deformation 
	retracts onto a surface contained in $\del X$.
	\item There is a universal bound on the number of components of $V_{i} \cap X$.
	\item There is a universal bound on the genus of each component of $V_{i} \cap X$.
\end{enumerate}
We obtain a certain cell decomposition of $V_{i} \cap X$.

Of course it is possible to triangulate $V_{i} \cap X$ with a bounded number of tetrahedra,
but that is not quite enough: the triangulations must agree on intersection in
order to yield a triangulation of $X$.  (Consider
a lens space: it is the union of two solid tori, but as
there are infinitely many distinct lens spaces, they require arbitrarily many tetrahedra.)
We triangulate $V_{i} \cap X$ in a way that agrees on
intersections using the cell decomposition mentioned in the
previous paragraph.

To get a bound on the number of 
tetrahedra, we observe that the faces of the  cell decomposition mentioned
above are totally geodesic.  This is used to bound the number of vertices,
which turns out to be the key for bounding the number of tetrahedra in our setting.

By contrast, when considering
the cell decomposition of the lens space $L(p,q)$ 
obtained by taking two solid tori and a meridian disk for each,
the number of vertices is not bounded; it equals 
the number of intersections between the disks, which is $p$.

\bigskip

\noindent{\bf A note on notation.}  Objects in $\mathbb{H}^{3}$
are denoted using tilde (for example, $\tilde s$ 
or $\wt C$) or using script lettering (for example, $\mathcal{A}$).  Constants denoted
by $C$ are universal (namely, $C$ as defined in Proposition~\ref{pro:decomposition} , 
$C_{3}$ as define in Lemma~\ref{lem:c3},
$C_{2}$ as define in Lemma~\ref{lem:c2}, 
$C_{1}$ as define in Lemma~\ref{lem:c1}, 
$C_{0}$ as define in Lemma~\ref{lem:c0},
and $\bar{C}_0$ as defined in Section~\ref{sec:ProofProp}).  Once defined
they are fixed for the remainder of the paper.
The constants $\mu$, $d$, $R$, and $D$ that were introduced in This section
are fixed throughout this paper.

\bigskip

\noindent{\bf Acknowledgment.}  We have benefitted from conversations and correspondences 
about Theorem~\ref{thm} with many experts and we are grateful to them all.  In particular, we thank  
Colin Adams, Joseph Maher, and Sadayoshi Kojima.  
We thank the anonymous referee for a careful reading of this paper and many helpful remarks.

\section{preliminaries}
\label{sec:prelims}

The notation of Section~\ref{sec:intro} is fixed for the remainder of this paper.

We assume familiarity with hyperbolic space $\mathbb{H}^3$ and its isometries, as 
well as the Margulis lemma.
The model of $\mathbb{H}^3$ we use is upper half space.  Given $\tilde x,\ \tilde y \in
\mathbb{H}^3$, we denote the closed geodesic segment connecting them by $[\tilde x,\tilde y]$.  
All manifolds considered are assumed to be orientable.  In a metric 
space, $N_d(\cdot)$ denotes the set of all points of distance at most $d$ from a 
given object.  The ball of radius $r$ centered at $x$ is
denoted $B(x,r)$.  The volume of a ball of radius $r$ in $\mathbb{H}^{3}$ is denoted
by $\mbox{Vol}(B(r))$.  We use the notation $\mbox{int}(\cdot)$ and
$\mbox{cl}(\cdot)$ for {\it interior} and {\it closure}.  

We fix $\mu > 0$ a Margulis constant for $\mathbb{H}^3$.  
By {\it hyperbolic manifold} $M$ we mean a complete, finite volume Riemannian
3-manifold locally isometric to $\mathbb H^{3}$.
The universal covering of a hyperbolic manifold $M$ is denoted
$\pi: \mathbb{H}^{3} \to M$; $\pi$ is called {\it the universal cover
projection}, or simply {\it the projection}, from $\mathbb H^3$ to $M$.  The {\it thick part} of $M$ is 
$$M_{[\mu,\infty)} = \{p \in M| \mbox{inj}_M(p) \geq \mu\}.$$ 
The {\it thin part} of $M$ 
is $$M_{(0,\mu]} = \mbox{cl}\{p \in M| \mbox{inj}_M(p) < \mu\} = \mbox{cl}(M \setminus M_{[\mu,\infty)}).$$
It is well known that $M = M_{(0,\mu]} \cup M_{[\mu,\infty)}$, $M_{(0,\mu]}$ is a disjoint union of
closed solid torus neighborhood of short geodesics and 
closed cusps, and
$M_{(0,\mu]} \cap M_{[\mu,\infty)}$  consists of tori.

\section{Voronoi Cells}
\label{sec:SetUp1}

Keep all notation as in the pevious sections, and recall that $N$ was the number of Voronoi cells
and $D = \min\{R,d\}$.
Since $\{x_1,\dots,x_N\}$ was chosen generically, we may assume that the Voronoi cells $\{V_{i}\}$
are transverse to each other and to $\del X$ (note that the Voronoi cells are a decomposition of $M$, 
not $X$, and $\del X \subset \mbox{int}(M)$).  
In the remainder of the paper, all our constructions
are generic and allow for small perturbation, and we always assume transversality (usually without 
explicit mention).
We bound $N$ in term of the volume of $M$:

\begin{lem}
  \label{lem:c3}
There exists a constant $C_3$ so that $N \leq C_3 \mbox{Vol}(M)$.
\end{lem}

\begin{proof}
For each $i$, $x_i \in X$, and hence by Proposition~\ref{pro:X}~(1),
$\mbox{inj}_M(x_i) > R \geq D$.  
Since $\{x_1,\dots,x_N\}$ is $D$-separated, for $i \neq j$, $d(x_{i},x_j) > D$.
Hence $\{B(x_i,D/2)\}_{i=1}^N$ is a set of balls disjointly
embedded in $M$, each of volume $\mbox{Vol}(B(D/2))$.  Thus 
$N \leq \mbox{Vol}(M)/ \mbox{Vol}(B(D/2))$; the lemma
follows by setting  
$$C_3 = 1/\mbox{Vol}(B(D/2)).$$
\end{proof}

The preimages of $\{x_{1},\dots,x_{N}\}$ in $\mathbb{H}^{3}$ gives rise to a Voronoi
cell decomposition of $\mathbb{H}^{3}$ in a similar manner to the cells in $M$.
It is convenient to fix one of these cells for each $i$:

\begin{notate}
\label{notate:TildeVi}
	\begin{enumerate}
	\item  For each $i$, fix a preimage of $x_{i}$, denoted $\tilde x_{i}$.
	\item $\wt V_{i}$ is the Voronoi cell corresponding to $\tilde x_{i}$, that is:
		$$\wt V_{i} = \{ \tilde p \in \mathbb H^3 | d(\tilde p,\tilde x_{i}) \leq d(\tilde p, \tilde q),
		\forall \tilde q \mbox{ so that } \ \pi(\tilde q) \in \{x_{1},\dots,x_{N}\}.\}$$

	\item For each $i$, the components of $V_{i} \cap X$ are denoted by
	$V_{i,j}$ ($j=1,\dots,n_{i}$), where $n_{i}$ is the number of the components of $V_{i} \cap X$.
	\item The preimage of $V_{i,j}$ in $\wt V_{i}$ is denoted $\wt V_{i,j}$, that is:
		$$\wt V_{i,j} = \{ \tilde p \in \wt V_{i} | \pi(\tilde p) \in V_{i,j}\}.$$
	\end{enumerate}
\end{notate}

\begin{lem}
\label{lem:ProjToSame}
If $\tilde p$, $\tilde p' \in \wt V_{i}$ project to the same point $p \in V_{i}$
then $d(\tilde p, \tilde x_{i}) = d(\tilde p', \tilde x_{i})$.
\end{lem}

\begin{proof}
Let $\tilde p$, $\tilde p'$ be points in $\wt V_{i}$ that project to the same point and
assume that $d(\tilde x_{i}, \tilde p) \neq d(\tilde x_{i}, \tilde p')$; say
$d(\tilde x_{i}, \tilde p) < d(\tilde x_{i}, \tilde p')$.
Since $\tilde p$ and $\tilde p'$ project to the same point, there is an isometry
$\phi \in \pi_{1}(M)$ so that $\phi(\tilde p) = \tilde p'$.  Let $\tilde x_{i}' = \phi(\tilde x_{i})$, 
for some $\tilde x_{i}' \in \pi^{-1}(x_{i})$.
Since $\phi$ acts freely $\tilde x_{i}' \neq \tilde x_{i}$.  We get: $d(\tilde x_{i}', \tilde p') = 
d(\phi^{-1}(\tilde x_{i}'),\phi^{-1}({\tilde p'})) = d(\tilde x_{i},{\tilde p}) <  d(\tilde x_{i}, \tilde p')$.
Hence $\tilde p' \not \in \wt V_{i}$, contradicting out assumption.  The lemma follows.
\end{proof}

In general, the distance between points in $V_{i}$ may be smaller than the distance
between their preimages in $\wt V_{i}$. 
However this is not the case when one of the points is $x_{i}$:

\begin{lem}
\label{lem:distance}
For any $\wt V_{i}$ and any $\tilde p \in \wt V_{i}$,
$ d(\tilde x_{i},\tilde p) = d(x_{i},p)$ (here $p = \pi(\tilde p)$).
\end{lem}

\begin{proof}
Of all paths from $x_{i}$ to $p$ in $M$, let $\beta$ be one that minimizes length 
(note that $\beta$ need not be unique).  First we claim that
$\beta \subset V_{i}$.  Suppose, for a contradiction, that this is not the case and let
$q \in \beta$ be a point not in $V_{i}$.  Then for some $j \neq i$, $d(q,x_{j}) < d(q,x_{i})$.
By connecting the shortest path from $p$ to $q$ to the shortest path
from $q$ to $x_j$ we obtain a path strictly shorter than $\beta$, showing 
that $d(p,x_{j}) < l(\beta) = d(p,x_{i})$.  Thus $p \not\in V_{i}$, a contradiction.  Hence
$\beta \subset V_{i}$.

Let $\tilde \beta$ be the lift of $\beta$ to $\mathbb H^{3}$ starting at $\tilde x_{i}$.
Then $\tilde \beta$ is a geodesic segment, say $[\tilde x_{i},\tilde p']$, 
for some $\tilde p'$ that projects to $p$.  Fix 
$\tilde q \in \{\pi^{-1}(x_{1}),\dots,\pi^{-1}(x_{n})\}$.  Then $[\tilde p',\tilde q]$
projects to $\pi([\tilde p',\tilde q])$, a path that connects $p$ to some point of
$\{x_1,\dots,x_{n}\}$.  By choice of of $\beta$ (and since paths have the same
length as their projections), $d(\tilde x_{i},\tilde p') = l(\tilde \beta) = l(\beta) \leq l(\pi([\tilde p',\tilde q]))
= l([\tilde p',\tilde q]) = d(\tilde p',\tilde q)$.  We conclude that $\tilde p' \in \wt V_{i}$.

We see that $d(x_{i},p) = l(\beta) = l(\tilde\beta) = d(\tilde x_{i},\tilde p')$.  Since $\tilde p$, 
$\tilde p' \in \wt V_{i}$,
by Lemma~\ref{lem:ProjToSame} $d(\tilde x_{i},\tilde p) = d(\tilde x_{i},\tilde p')$;
the lemma follows.
\end{proof}

A {\it convex polyhedron} is the intersection of half spaces in $\mathbb H^{3}$. 
Note that a convex polyhedron is not required to be of
bounded diameter or finite sided (that is, the intersection of finitely many half spaces).

\begin{lem}
\label{lem:voronoi}
$\wt V_{i}$ is a convex polyhedron that projects onto $V_{i}$
\end{lem}

\begin{proof}   It is immediate that $\wt V_{i}$ is a convex polyhedron.

Given any $\tilde p \in \wt V_{i}$, $[\tilde x_{i},\tilde p]$ is the shortest geodesic 
from $\tilde p$ to any preimage of $\{x_{1},\dots,x_{N}\}$.  The projection of 
$[\tilde x_{i},\tilde p]$ is the shortest geodesic from the projection of $\tilde p$
to $\{x_{1},\dots,x_{N}\}$.  It follows easily that $\pi(\tilde p) \in V_{i}$.  As $\tilde p$ 
was an arbitrary point of $\wt V_{i}$, we see that
$\wt V_{i}$ projects into $V_{i}$.

Conversely, given any $p \in V_{i}$, let $\beta$ be the shortest geodesic from
$\{x_{1},\dots,x_{N}\}$ to $p$.  Then $\beta$ connects $x_{i}$ to $p$.
Let $\tilde \beta$ be the unique lift of $\beta$ that starts at $\tilde x_{i}$,
and denote its terminal point by $\tilde p$.
Similar to the argument of the proof of Lemma~\ref{lem:distance}, $\tilde \beta$ is the
shortest geodesic connecting any preimage of $\{x_{1},\dots,x_{n}\}$ to
$\tilde p$, showing that $\tilde p \in \wt V_{i}$.
Hence $p$ is in the image is $\wt V_{i}$;  As $p$ was an arbitrary point of $V_{i}$, we see that
$\wt V_{i}$ projects onto $V_{i}$.
\end{proof}

%

\noindent {\bf Decomposition of $V_{i}$.} By Lemma~\ref{lem:voronoi}, the boundary
of $\wt V_{i}$ is decomposed into faces, edges and vertices.  By the same lemma,
it projects into $V_{i}$.  The images of this faces, edges and vertices from
the decomposition of $V_{i}$ that is the basis for our work in the next section.
Note that some faces of $\wt V_{i}$ are identified, and the
corresponding faces of $V_{i}$ are contained in the interior, not boundary, of $V_{i}$.
(We will show  in Lemma~\ref{lem:HB}~(3) that faces in the interior of $V_{i}$ 
are contained in $M \setminus X$, and they will play no role in our construction.)
We remark that this is not the final decomposition: in the next section we will add more faces, 
edges and vertices to the decomposition.

\section{Decomposing $X$}
\label{sec:SetUp2}

We first define:

\begin{dfn}
\label{dfn:cone}
Fix $r>0$ and a geodesic $\tilde\gamma \subset \mathbb{H}^3$.  Let $\wt C =
\{p \in \mathbb{H}^3 | d(p,\tilde \gamma) \leq r\}$.  We call $\wt C$ a {\it cone}
about $\gamma$, or simply a cone\footnote{In the upper half space model, if $\tilde \gamma$
is a Euclidean vertical straight ray from $\tilde p_{\infty}$ in the $xy$-plane,
then $\wt C$ is the cone of all Euclidean straight rays from $\tilde p_{\infty}$ that
form angle at most $\alpha$ (for some $\alpha$) with $\tilde \gamma$.
If $\tilde\gamma$ is a semicircle then $\wt C$ looks more like a banana.},
and $\tilde \gamma$ the {\it axis} of $\wt C$.
The set  $\{p \in \mathbb{H}^3| d(p,\gamma) \geq r \}$ is called {\it exterior} 
of $\wt C$, denoted $\wt E$.    
\end{dfn}

The reason we look at cones is that if $V$ is a solid torus component of $\mbox{cl}(M \setminus X)$ 
and  $\gamma$ its core geodesic, then $\pi^{-1}(V)$ is a cone and $\pi^{-1}(\gamma)$
its axis.  It can be seen directly that the intersection of a geodesic and a cone is 
a (possibly empty) connected set; hence cones are convex.  If $V$ is a cusp component of $\mbox{cl}(M \setminus X)$, then
its preimage is a horoball which is also convex.
Below, we often use the fact that the
every component of the preimage of $\mbox{cl}(M \setminus X)$  is convex.

\begin{lem}
  \label{lem:ThinInVoronoi}
For any $i$, $V_i \cap \mbox{\rm cl}(M \setminus X)$ is connected.
\end{lem}

\begin{proof}
The number of times $V_i$ intersects $\mbox{cl}(M \setminus X)$ is at most the
number of times $\widetilde{V}_i$ intersects the preimage of $\mbox{cl}(M
\setminus X)$.  Since $\widetilde{V}_i$ and any component of the preimage of
$\mbox{cl}(M \setminus X)$ are both convex, their intersection is connected.
Thus all we need to show is that $\widetilde{V}_i$ intersects only one
component of the preimage of $\mbox{cl}(M \setminus X)$.

Suppose this is not the case, and let $\tilde\alpha$ be the shortest arc in 
$\widetilde{V}_{i}$ that connects distinct components of the preimage of 
$\mbox{cl}(M \setminus X)$.  Since $\wt V_{i}$ is convex, 
$\tilde \alpha$ is a geodesic.  
Since $\tilde\alpha$ connects {\it distinct} components of the preimage 
of $\mbox{cl}(M \setminus X)$, some point on $\tilde \alpha$ projects
into $M_{[\mu,\infty)}$.  Let $\alpha = \pi(\tilde\alpha)$.  
Recall that the distance from $\del X$ to $M_{[\mu,\infty)}$ 
is $d$.  We conclude that $l(\tilde\alpha) = l(\alpha) > 2d \geq 2D$.  
Thus the distance between the endpoints of $\tilde\alpha$ is greater than $2D$, and by 
the triangle inequality, for some point $\tilde p \in \tilde\alpha$, 
$d(\tilde x_{i},\tilde p) > D$.  Let $p$ be the image of $\tilde p$.
By Lemma~\ref{lem:distance}, $d(x_i,p) = d(\tilde x_{i},\tilde p)$
and by Lemma~\ref{lem:voronoi} $p \in V_{i}$.
Hence by construction of the Voronoi cells, for any $j$, $d(x_{j},p) \geq d(x_{i},p) > D$.
Thus $\{x_{1},\dots,x_{n},p\} \subset X$ is a $D$-separated
set, contradicting maximality of $\{x_1,\dots,x_{N}\}$.
The lemma follows.
\end{proof}

In the  next lemma we bound the number of faces of $V_{i}$ that intersect $X$
and study that intersection:

\begin{lem}
  \label{lem:c2}
  The following two conditions hold:
  \begin{enumerate}
	\item There exists a constant $C_2$ so that for every $i$, $1 \leq i \leq N$, 
	the number of faces of $V_i$ that intersect $X$ is at most $C_2$.
  	\item For each $i$ and every face $F$ of $V_{i}$, $F \cap X$ is either empty, 
	or a single annulus, or a collection of disks. 
  \end{enumerate}
\end{lem}

\begin{proof}
Let $V_i$ be a Voronoi cell, $F$ a face of $V_{i}$ so that $F \cap X \neq \emptyset$,
and $p \in F \cap X$.  Let $\tilde p \in \wt V_{i}$ be a preimage of $p$ ($\tilde p$
exists by Lemma~\ref{lem:voronoi}) and let $\wt F$ be a face of $\wt V_{i}$
containing $\tilde p$. 
Let $\tilde x$ be the preimage of $\{x_{1},\dots,x_{N}\}$
that is contained in the cell adjacent to $\tilde
F$ on the opposite side from $\widetilde{V}_i$.  
By Lemma~\ref{lem:distance},  $d(\tilde x_{i},\tilde p) = d(x_{i},p)$.  Similar
to the argument of the proof of Lemma~\ref{lem:ThinInVoronoi}, maximality
of $\{x_{1},\dots,x_{N}\}$ implies that $d(x_{i},p) < D$.
We conclude that $d(\tilde x_{i},\tilde p) <D$, and similarly
$d(\tilde x,\tilde p) <D$.
By the triangle inequality, $d(\tilde x_i,\tilde x) < 2D$.    

For each face $F$ of $V_{i}$ with $F \cap X \neq \emptyset$, consider the cell
adjacent to $\wt V_{i}$ along $\wt F$ as constructed above.
The balls of radius $D/2$ centered at the preimages of $\{x_1,\dots,x_{N}\}$ in
these cells are
disjointly embedded and their centers are no further than $2D$
from $\tilde x_{i}$, so these balls are contained in $B(\tilde x_{i},2.5D)$.
Thus~(1) follows by setting
$$C_2 = \mbox{Vol}(B(2.5D))/\mbox{Vol}(B(D/2)). $$

For~(2), fix $V_{i}$ and $F$ a face of $V_{i}$.  Let $\wt F$ be the face of $\wt V_{i}$ that 
projects to $F$.  Since $\wt V_{i}$ is a convex polyhedron, $\wt F$ is a totally geodesic convex
polygon.  By Lemma~\ref{lem:ThinInVoronoi}, $\wt F$ intersects at most one component of
the preimage of $M \setminus X$, and by convexity of that component and of $\wt F$,
the intersection is either empty or a disk.  We see that one of the following holds:
	\begin{enumerate}
	\item When the intersection is empty: then the intersection of $\wt F$ with the preimage of 
	$X$ is $\wt F$ (and hence a disk).
	\item  When the intersection is a disk contained in $\mbox{int}(\wt F)$: then the intersection of $\wt F$ with the preimage of 
	$X$ is an annulus.
	\item  When the intersection is a disk not contained in $\mbox{int}(\wt F)$: then the intersection of $\wt F$ with the preimage of 
	$X$ is a collection of disks.
	\end{enumerate}

We claim that the intersection of $\wt F$ with the preimage of $X$ projects homeomorphically
onto its image.   Otherwise, there are
two points $\tilde p_{1}$, $\tilde p_{2} \in \wt F$ that project to the same point $p \in F \cap X$.
By maximality of $\{x_{1},\dots,x_{N}\}$, $d(x_{i},p) < D$.  By Lemma~\ref{lem:distance},
$d(\tilde x_{i},\tilde p_{1})$, $d(\tilde x_{i},\tilde p_{2}) = d(x_{i},p)$.
The shortest path from $\tilde p_{1}$ to $\tilde p_{2}$ that goes through $\tilde x_{i}$ projects 
to an essential closed path that contains $x_{i}$ and has length less than $2D$.  But then
$\mbox{inj}_{M}(x_{i}) < D \leq R$, contradicting Proposition~\ref{pro:X}~(1). 
Thus the intersection of the preimage of $X$ with $\wt F$ 
projects homeomorphically and ~(2) follows.
\end{proof}

We consider the intersection of an edge $e$ of $V_{i}$ with $X$.  We call the
components of $e \cap X$ {\it segments}.
In the next lemma we bound the number of segments:

\begin{lem}
  \label{lem:c1}
There exists a constant $C_1$ so that for every $i$, $1 \leq i \leq N$, the
number of segments from the intersection of edges of $V_i$ with $X$
is at most $C_1$. 
\end{lem}

\begin{proof}
Fix $i$ and $e$ an edge of $V_i$.    We first show that $e$ contributes at
most two segment. If $e \subset X$ then it contributes exactly one segment
and if $e \cap X = \emptyset$ then it contributes no segment.  Otherwise, 
let $\tilde e$ be a lift of $e$ that is in $\widetilde{V}_{i}$.
By Lemma~\ref{lem:ThinInVoronoi}, $\tilde e$ intersects at most one component  of
the preimage of  $\mbox{cl}(M \setminus X)$.   
Since $\tilde e$ and any component of the preimage of $\mbox{cl}(M \setminus X)$ 
are both convex, their intersection is convex and hence 
connected.  Thus the intersection of $\tilde e$ and the preimage
of $\mbox{cl}(M \setminus X)$ is connected, and projecting to $M$ we see that the intersection
of $e$ and $\mbox{cl}(M \setminus X)$ is connected as well.  Thus $e$ contributes at most 2
segments.

Since $\widetilde{V}_{i}$ is convex, the intersection of 2 faces of $\del \widetilde{V}_{i}$ is
at most one edge.  Hence the number of edges is bounded above by the the
number of pairs of faces, $\frac{1}{2}C_2(C_2-1)$ (using Lemma~\ref{lem:c2}).
The number of edges of $V_{i}$ is no larger;
Lemma~\ref{lem:c1} follows by setting  
$$C_1 = C_2 (C_2-1).$$
\end{proof}

\begin{lem}
  \label{lem:c0}
There exists a constant $C_0$ so that for every $i$, $1 \leq i \leq N$, the
number of vertices of $V_i$ that lie in $X$ is at most $C_0$. 
\end{lem}

\begin{proof}
Each segment contributes at most 2 vertices. Lemma~\ref{lem:c0} follows by setting  
$$C_0 = 2C_1.$$
\end{proof}

\begin{dfn}
\label{dfn:s-convex}
Let $\wt C$ be a cone, $\tilde \gamma$ its axis, and $\wt E$ its exterior
(recall Definition~\ref{dfn:cone}).  Fix $\tilde s \in \wt C$.   
We say that a set $\widetilde{K} \subset \wt E$ is {\it $\tilde s$-convex} 
if for any $\tilde p \in \widetilde{K}$,  $[\tilde p,\tilde s] \cap \widetilde{E}$ is
contained in $\widetilde{K}$.  
\end{dfn}

\begin{lem}
  \label{lem:deformation}
Let $\wt K \subset \wt E$ be an $\tilde s$-convex set (for some $\tilde s \in \wt{C}$).
Then there exists a deformation retract from $\wt K$ onto  $\wt K \cap \del \wt E$.
\end{lem}

\begin{proof}
Fix $\tilde p \in \wt K$.  
Since $\wt C$ is convex, $[\tilde p,\tilde s] \cap \widetilde{C}$ 
is an interval, say $[\tilde r, \tilde s]$, and $\wt E \cap [\tilde r, \tilde s] = [\tilde p,\tilde r]$.  
Since $\wt K$ is $\tilde s$-convex, $[\tilde p,\tilde r] \subset \wt K$.  We move $\tilde p$
along $[\tilde p, \tilde r]$ from its original position to $\tilde r \in\wt K \cap \del \wt E$ in
constant speed.  It is easy to see that this is a deformation retract.

\end{proof}

\begin{notate}
\label{notate:f}
With the notation of the previous lemma, we define 
$f:\wt K  \to \wt K \cap \del \wt E$ to be $f(\tilde p) = \tilde r$.
\end{notate}

Recall the definition of $\wt V_{i}$, $V_{i,j}$, and $\wt V_{i,j}$ from Notation~\ref{notate:TildeVi}.
Recall also that $n_{i}$ was the number of components of $V_{i} \cap X$ 
from  Notation~\ref{notate:TildeVi}~(3).
\begin{lem}
  \label{lem:HB}
The following conditions hold:
\begin{enumerate}
\item For each $i$, $n_{i} \leq C_0$.  
\item For each $i,$ $j$, if $\wt V_{i,j} \neq \wt V_i$, there is a
cone $\wt C$, a component of the preimage of $\mbox{cl}(M \setminus X)$, 
so that $\wt V_{i,j}$ is $\tilde s$ convex for any
point $\tilde s \in \wt V_{i} \cap \wt C$.
\item The projection of $\cup_{j=1}^{n_{i}}\wt V_{i,j}$ to 
$\cup_{j=1}^{n_{i}} V_{i,j}$ is a diffeomorphism.
\item For each $i,$ $j$, $V_{i,j}$ is a handlebody.      
\end{enumerate}

\end{lem}

\begin{proof}
Fix $i$.  It is easy to see that each component of $V_{i} \cap X$ 
must contain a vertex of $V_i$.
Applying Lemma~\ref{lem:c0} we see that there are at most $C_0$ such
components.  This establishes~(1). 

Any component of the preimge of $\mbox{cl}(M \setminus X)$ is a
cone.  Assuming that $V_{i,j} \neq V_i$, by Lemma~\ref{lem:ThinInVoronoi}
there exists a unique such component, say $\wt C$, that intersects $\wt V_{i}$.  Fix a point
$\tilde s \in \wt C \cap \wt V_{i}$.   Fix $\tilde p \in \wt V_{i,j}$.
Convexity of $\wt V_{i}$ implies that $[\tilde p, \tilde s] \subset \wt V_{i}$.  Convexity
of $[\tilde p, \tilde s]$ and $\wt C$ implies that $[\tilde p, \tilde s]  \cap \wt C$ is an
interval, say $[\tilde r, \tilde s]$, and therefore 
$[\tilde p,\tilde s] \cap (\wt V_{i} \setminus \mbox{int}\wt C) = [\tilde p, \tilde r]$. 
Since $\wt V_{i,j}$ is connected,
$[\tilde p,\tilde s] \cap \wt V_{i,j} = [\tilde p,\tilde r] \subset \wt V_{i,j}$.
This establishes~(2).

For~(3), it is easy to see that all we need to show is that the projection 
$\cup_{j=1}^{n_{i}} \wt V_{i,j} \to \cup_{j=1}^{n_{i}} V_{i,j}$ is one-to-one.  
Assume not (this is similar to Lemma~\ref{lem:c2}~(2)); then there exist $\tilde p_{1}, \ \tilde p_{2}
\in \cup_{j=1}^{n_{i}} \wt V_{i,j}$ that project to the same point $p \in \cup_{j=1}^{n_{i}} V_{i,j}$.   
Then the shortest path from $\tilde p_{1}$ to $\tilde p_{2}$ that goes through $\tilde x_{i}$ projects 
to an essential closed path that contains $x_{i}$, and has length less than $2D$.  But then
$\mbox{inj}_{M}(x_{i}) < D \leq R/2$, contradicting Proposition~\ref{pro:X}~(1). 
This establishes~(3).

If $V_i \subset X$ then $V_i \cap X$ is a ball and~(4) follows.  Otherwise,~(4)
follows from~(2), Lemma~\ref{lem:deformation}, and~(3).
\end{proof}

We denote $V_{i,j} \cap \del X$ by $P_{i,j}$.  We bound $g(V_{i,j})$, the genus of $V_{i,j}$: 

\begin{lem}
  \label{lem:genus}
$g(V_{i,j}) \leq C_1$.
\end{lem}

\begin{proof}
If $V_{i,j} = V_{i}$ then it is a ball and there is nothing to show.  Assume this
is not the case.  Then by Lemmas~\ref{lem:HB} and ~\ref{lem:deformation} 
$V_{i,j}$ deformation retracts onto
$P_{i,j}$.  Hence $\mbox{cl}(\del V_{i,j} \setminus P_{i,j})$ is homeomorphic 
to $P_{i,j}$, and is a
$g(V_{i,j})$-times punctured disk.  The faces of $V_i$ induce a decomposition
on $\mbox{cl}(\del V_{i,j} \setminus P_{i,j})$.  By Lemma~\ref{lem:c2}~(2),
each face of $\mbox{cl}(\del V_{i,j} \setminus P_{i,j})$ is a disk or 
an annulus; in
particular the Euler characteristic of each such component is non-negative.
Denote the faces of $\mbox{cl}(\del V_{i,j} \setminus P_{i,j})$ by $\{F_k\}_{k=1}^{k_{0}}$, the number of
edges by $e$, and the number of vertices by $v$.  Note further, that the edges of 
$\mbox{cl}(\del V_{i,j} \setminus P_{i,j})$ come in two types, edges
in the interior of $\mbox{cl}(\del \wt V_{i,j} \setminus P_{i,j})$
(say $e_{int}$ of them) and edges on its boundary
(say $e_{\del}$ of them).  Similarly, $v_{int}$ (resp. $v_{\del}$)
denotes the number of vertices in the interior (resp. boundary) of $\mbox{cl}(\del \wt V_{i,j} \setminus P_{i,j})$.
Since  the boundary of $\mbox{cl}(\del V_{i,j} \setminus P_{i,j})$  consists of circles,
$e_{\del} = v_{\del}$.  Since $e_{int}$ is the number of segments on $V_{i,j} \cap X$,
by Lemma~\ref{lem:c1}, $e_{int} \leq C_{1}$.  An Euler characteristic calculation gives:

\begin{align*}
   1 - g(V_{i,j}) & = \chi(\mbox{cl}(\del V_{i,j} \setminus P_{i,j})) & \mbox{\rm since }
   		\mbox{cl}(\del V_{i,j} \setminus P_{i,j}) \mbox{ \rm is a } g(V_{i,j})\mbox{\rm -times 
		 disk}\\
   & = (\Sigma_{k=1}^{k_{0}} \chi(F_k)) -e +v \\
   & = (\Sigma_{k=1}^{k_{0}} \chi(F_k)) -(e_{int} + e_{\del}) +(v_{int} + v_{\del}) \\
   & = (\Sigma_{k=1}^{k_{0}} \chi(F_k)) - e_{int} +v_{int} & e_{\del} = v_{\del} \\
   & \geq  - e_{int} + 1 &  \Sigma_{k=1}^{k_{0}} \chi(F_k) \geq 0  \mbox{ \rm and } v_{int} > 0 \\
   & \geq -C_1 + 1. & \mbox{\rm Lemma~\ref{lem:c1}}
\end{align*}
The lemma follows.
\end{proof}

We use the notation $\del_{\infty}$ for the limit points at infinity.  In particular, if $\wt L$
is a totally geodesic plane then $\del_{\infty}\wt L$ 
is a simple closed curve and if $\wt C$ is a cone then 
$\del_{\infty}\wt C$ consists of two points.

\begin{lem}
  \label{lem:arc}
Let $\wt C$ be a cone, $\tilde s \in \wt C$ a point not on the axis of $\wt C$.  Let 
$\tilde\beta \subset \del \wt C$ be an arc and $\tilde p \in \tilde\beta$,
$\tilde q \in \del \wt C$ points.  Assume that the totally geodesic plane
that contains $\tilde s$ and $\del_{\infty} \wt C$ intersects $\tilde\beta$ in
a finite set of points.

Then after an arbitrarily small perturbation of $\tilde p$ in $\tilde \beta$, there exists a
closed arc $\tilde\alpha \subset \del \wt C$ connecting  $\tilde p$ and $\tilde q$, so 
that $\tilde \alpha$ and $\tilde s$ are contained in a totally geodesic plane.
\end{lem}

\begin{rmk}
The condition on $\tilde\beta$ is generic, and since we allow 
for small perturbations in our 
construction we will always assume it holds.
\end{rmk}

\begin{proof}
Let $\wt L$ be a totally geodesic plane containing $\tilde p$,
$\tilde q$, and $\tilde s$.  We prove Lemma~\ref{lem:arc} in two cases:



\medskip

\noindent{\bf Case One.}  $\del_{\infty} \wt C \not\subset \del_{\infty} \wt L$.
The reader can easily verify that in this case  $\wt L \cap \del \wt C$ is connected.
Then we take the arc $\tilde \alpha$ 
to be a component of $\wt L 
\cap \del \wt C$ that connects $\tilde p$ to $\tilde q$.  

\medskip

\noindent{\bf Case Two.}   $\del_{\infty} \wt C \subset \del_{\infty} \wt L$.  Equivalently, 
$\tilde\gamma \subset \wt L$, where  $\tilde\gamma$ denotes the axis of $\wt C$.
Since $\tilde s \not\in \tilde\gamma$, $\wt L$ is the {\it unique}
totally geodesic plane that contains both $\tilde s$ and $\tilde \gamma$.  
By assumption, $\tilde\beta \cap \wt L$ is a finite set.  
By perturbing $\tilde p$ slightly in $\tilde\beta$ we reduce the problem to Case One.
The lemma follows.
\end{proof}

In the following lemma we construct the main tool we will use for cutting $V_{i,j}$ into balls.
In that lemma, $\wt C$ is a cone and $\tilde s \in \wt C$ a point.
A collection of simple closed curves on $\del \wt C$ is called {\it generic} if 
it intersects any totally geodesic plane containing $\tilde s$ in a finite collection of points.
As remarked after Lemma 4.11, we will always assume it holds.

\begin{lem}
  \label{lem:complex}
Let $\wt C$ be a cone, $\tilde s \in \wt C$ a point
not on the axis of $\wt C$, and 
$\mathcal{C} \subset \del \wt C$ a collection of $n+1$ disjoint generic simple closed curves, for some $n \geq 0$.  

Then there exists a graph $\mathcal{A} \subset \del \wt C$ with the following properties:
	\begin{enumerate}
	\item $\mathcal{A}$ has at most $2n-1$ edges.
	\item For every edge $\tilde e$ of $\mathcal{A}$, $\tilde e$ and $\tilde s$ are contained in a single totally 
	geodesic plane.
	\item $\mathcal{C} \cup \mathcal{A}$ is a connected trivalent graph. 
	\item  The graph obtained by removing any edge of $\mathcal{A}$ from 
	$\mathcal{A} \cup \mathcal{C}$ is disconnected.
	\end{enumerate}
\end{lem}
  
\begin{proof}

We induct on $n$.  If $n=0$ there is nothing to prove.

Assume $n>0$ and let $\tilde c$ be a component of $\mathcal{C}$.  Let $\mathcal{C}' = \mathcal{C}
\setminus \tilde c$.  By the induction hypothesis, there exists a graph $\mathcal{A}' \subset \del \wt C$
with at most $2n-3$ edges, so that every edge of 
$\mathcal{A}'$ is contained in a totally geodesic plane that contains $\tilde s$,
and $\mathcal{C}' \cup \mathcal{A}'$ is a connected trivalent graph.

\medskip

\noindent{\bf Case One.}  $\tilde c \cap (\mathcal{C}' \cup \mathcal{A}') = \emptyset$.
Fix $\tilde p \in  \tilde c$ and $\tilde q \in \mathcal{C}' \cup \mathcal{A}'$
so that $\tilde q$ is not a vertex.  
Since $\tilde c$ is generic, by Lemma~\ref{lem:arc} 
after a small perturbation of $\tilde p$ in $\tilde c$,
there exists an arc $\tilde \alpha'$ connecting $\tilde p$
and $\tilde q$ so that $\tilde \alpha'$ and $\tilde s$ are contained in a totally
geodesic plane.   Since the perturbation was generic we may assume that $\tilde
\alpha'$ is transverse to $\mathcal{C}' \cup \mathcal{A}' \cup \tilde c$.

Let $\tilde \alpha$ be a component of $\tilde \alpha'$ cut open along 
the points of $\tilde \alpha' \cap (\mathcal{C}' \cup \mathcal{A}' \cup \tilde c)$ that connects
$\tilde c$ to $\mathcal{C}' \cup \mathcal{A}'$.  Since the perturbation was
generic we may assume that the endpoint of $\tilde\alpha$ on $\mathcal{A}' \cup 
\mathcal{C}'$ is not a vertex,
so that $\mathcal{C} \cup \mathcal{A}' \cup \tilde \alpha$ is a connected trivalent
graph.  The lemma follows in Case One by setting $\mathcal{A} = \mathcal{A}' \cup
\tilde \alpha$.

\medskip

\noindent{\bf Case Two.} $ \tilde c \cap (\mathcal{C}' \cup \mathcal{A}') \neq \emptyset$.  
Let $\mathcal{A}''$ be the graph obtained from $\mathcal{A}'$ by adding a vertex at every point of 
$\mathcal{A}' \cap \tilde c$.  Note that there is no bound on the number of
edges of $\mathcal{A}''$, and the the vertices of $\mathcal{C} \cup \mathcal{A}''$ have valence 3
or 4 (the vertices of valence 4 are $\mathcal{A} ' \cap \tilde c$).  
Clearly, $\mathcal{C} \cup \mathcal{A}''$ is connected.   

\smallskip

\noindent {\bf Step One.}  Let $\tilde e$ be an edge of $\mathcal{A}''$ so that the graph obtained by
removing $\tilde e$ from $\mathcal{C} \cup \mathcal{A}''$ is connected.  
We remove $\tilde e$.   

\smallskip

\noindent{\bf Step Two.}  Note that as after Step One we may have a vertex, say $\tilde v$,
of valence 2.
Let $\tilde e_i$ ($i=1,2$) be the other two edges incident to $\tilde v$;
denote the endpoints of $\tilde e_{i}$ by $\tilde v$ and $\tilde v_i$.  
We remove $\tilde e_{1}$ and $\tilde e_{2}$.  If the graph obtained 
is disconnected, then it consists of  two components, one containing $\tilde v_{1}$ and one
containing $\tilde v_{2}$.  As in Case One, we construct an arc to connect the two components.

Step Two may produce a new vertex of valence 2.
We iterate Step Two.  This process
reduces the number of edges and so terminates; when it does, we obtain
a connected graph with no vertices of valence 2.   

We now repeat Step One (if possible).  After every application of Step One we repeat
Step Two (if necessary).  Step One also reduces the number of edges, so it terminates.  
When it does, we obtain a graph
(still denoted $\mathcal{A}''$) so that
$\mathcal{C} \cap \mathcal{A}'' $ is connected, but removing any edge of $\mathcal{A}''$ 
disconnects it.

By construction, the vertices of $\mathcal{C} \cup \mathcal{A}''$ 
have valence 3 or 4.  If a vertex has valence
4, we choose an edge adjacent to it from $\mathcal{A}''$.  Perturbing the endpoint of
this edge and applying Lemma~\ref{lem:arc}, 
we obtain two vertices of valence 3.  We iterate
this process.  Since this process reduces the number of vertices of valence 4, it will
terminate.  The graph obtained is denoted $\mathcal{A}$.

By construction, conditions~(2),~(3) and~(4) of Lemma~\ref{lem:complex} hold.  All that
remains is proving:

\medskip

\noindent{\bf Claim.}  $\mathcal{A}$ has at most $2n-1$ edges.

\medskip

\noindent{Proof of claim:} Let $\Gamma$ be the graph obtained from $\mathcal{C} \cup \mathcal{A}$
by identifying  every component of  $\mathcal{C}$ to a single point.
Note that the vertex set of $\Gamma$ has $n+1$ vertices that correspond
to the components of $\Gamma$, and extra vertices from the vertices
of $\mathcal{A}$ that are disjoint from $\mathcal{C}$; these vertices
all have valence 3.  The edges of 
$\Gamma$ are naturally in 1-1 correspondence with the edges of
$\mathcal{A}$; thus to prove the claim all we need to show is that 
$\Gamma$ has at most $2n-1$ edges.

It is easy to see that $\Gamma$ is connected because $\mathcal{C} \cup \mathcal{A}$
is.  Moreover, if there is any edge $e$ of $\Gamma$ so that the graph 
obtained from $\Gamma$ by removing $e$ is connected, then the
graph obtained from $\mathcal{C} \cup \mathcal{A}$ by removing the
corresponding edge is connected as well; this contradict our construction.
Hence $\Gamma$ is a tree, with $n+1$ vertices of arbitrary valence,
and all other vertices have valence 3.
In particular,  $\Gamma$ has at most $n+1$ vertices 
of vertices of valence 1 or 2.  We will use the following claim:

\bigskip

\noindent {\bf Claim.}  Let $G = (V,E)$ be a finite tree with vertex set $V$ and edge set $E$ and
with $k \geq 2$ vertices of valence 1 or 2.
Then $G$ has at most $2k-3$ edges.

\bigskip

We prove the claim by induction on $k$.  If $k=2$, it is easy to see that $G$ is a single 
edge, and indeed $1 = 2 \cdot 2 -3$.  Assume from now on that $k > 2$.

It is well known that every finite tree has a leaf (that is, a vertex of valence 1).  Let
$v \in V$ be a leaf and $(v,v') \in E$ the only edge containing $v$.  Consider 
$G' = (V - \{v\}, E - \{(v,v')\})$.  There are four cases, depending on the valence of $v'$
as a vertex of $G'$:
	\begin{enumerate}
	\item The valence of $v'$ is zero: then $G$ is a single edge, contrary to our assumption.
	\item The valence of $v'$ is one: that is, $v'$ is a leaf of $G'$.  Note that in this case both $v$
	and $v'$ have valence 1 or 2 in $G$, and we see that $G'$ has exactly $k-1$ vertices of valence
	1 or 2.  By induction $G'$ has at most $2(k-1) - 3 = 2k - 5$ edges.  Since $G$ has exactly one more
	edge than $G'$, $G$ has at most $2k-4$ edges in this case.
	\item The valence of $v'$ is two:  In this case, the number of vertices of valence 1 or 2 in 
	$G'$ is exactly $k$ (note that $v'$ has valence 3 in $G$).  Let $(v',v'')$ and $(v',v''')$ be the two edges 
	adjacent to $v'$.   Let $G''$ be the graph obtained from $G'$ by removing $v'$ from the vertex set
	and $(v',v'')$, $(v',v''')$ from the edge set, and add the edge $(v'',v''')$.  It is easy
	to see that $G''$ is a tree with exactly $k-1$ vertices of valence 1 or 2.  By induction $G''$
	has at most $2(k-1) -3$ edges.  Since $G'$ has one more edge than $G''$ and $G$ has one
	more edge than $G'$, $G$ has at most $2k-3$ edges as desired.
	\item The valence of $v'$ is at least three: then $G'$ has exactly $k-1$ vertices of valence 1 or 2.
	Similar to the above, we see that $G$ has at most $2k-4$ vertices in this case.
	\end{enumerate}

This proves the claim.  

To establish~(1), we use the claim and the fact that $\Gamma$ is a tree with at most 
$n+1$ vertices of valence 1 or 2, and see that the number of edges in $\Gamma$ is at most $2(n+1) - 3 = 2n-1$.

This completes the proof of Lemma~\ref{lem:complex}.

\end{proof}

Next, we prove the existence of totally geodesic disks that cut $V_{i,j}$ into balls.  We note that the disks 
may not be disjoint.  The precise statement is:

\begin{lem}
\label{lem:K}
For any $V_{i,j}$ there exists a 2-complex ${K}_{i,j} \subset V_{i,j}$ 
so that the following hold:
	\begin{enumerate}
	\item $V_{i,j}$ cut open along ${K}_{i,j}$ is a single ball.  
	\item The faces of ${K}_{i,j}$ are totally geodesic disks.
	The edges of ${K}_{i,j}$ have valence 3.
	\item All the vertices are on $\del V_{i,j}$.
	\item ${K}_{i,j}$ has at most $2C_1 -  1$ faces and $4C_{1}-2$ edges in the interior of $V_{i,j}$.
	\end{enumerate}
\end{lem}

\begin{proof}
If $V_{i,j} = V_{i}$ then it is a ball and there is nothing to prove.
Assume this is not the case.  Then $V_{i} \neq V_{i,j}$, and hence  $V_{i} \cap \mbox{cl}(M\setminus X)
\neq \emptyset$; by Lemma~\ref{lem:ThinInVoronoi}, $\wt V_{i}$ intersects 
exactly one preimage of 
$\mbox{cl}(M\setminus X)$, say $\wt C$.  Recall that $\wt C$ is a cone.  
We first establish conditions
analogous to~(1)--(4) for $\wt V_{i,j}$.

Let $\wt P_{i,j}$ denote $\wt V_{i,j} \cap \del \wt C$.  
By Lemma~\ref{lem:deformation}, $\wt V_{i,j}$ deformation retracts 
onto $\wt{P}_{i,j}$; hence $\wt{P}_{i,j} \subset \del \wt{C}$
is a connected, planar surface. 
 
Let $\tilde s \in \wt V_{i} \cap \wt C$ be a point not on the axis of $\wt C$.  The Voronoi cells 
were constructed around generic points $\{x_{i}\}$.  Therefore, after perturbing $\tilde s$ slightly
if necessary, 
$\del \wt P_{i,j} \subset \del \wt C$ is a generic collection of circles, and Lemma~\ref{lem:complex}
applies to give a  graph, denoted $\mathcal{A}_{i,j}$, so that $\del \wt P_{i,j}$ and
$\mathcal{A}_{i,j}$ fulfill the conditions of Lemma~\ref{lem:complex}.  
It follows easily from Lemma~\ref{lem:complex}~(4) that  $\mathcal{A}_{i,j} \subset \wt P_{i,j}$.

Set 
$\mathcal{K}_{i,j}$ to be $f^{-1}(\mathcal{A} _{i,j})$, where the function $f$ is described in 
Notation~\ref{notate:f}.  By construction, for every edge $\tilde e$ of $\mathcal{A}_{i,j}$,
$f^{-1}(\tilde e)$ is the intersection of the totally geodesic plane containing $\tilde e$
and $\tilde s$ with $\wt V_{i,j}$.  Since $f$ is a deformation retract, $f^{-1}(\tilde e)$
is a disk.  These disks are the {\it faces} of $\mathcal{K}_{i,j}$; thus 
the faces of $\mathcal{K}_{i,j}$ are totally
geodesic disks.
By Lemma~\ref{lem:complex}, $\del\wt{P}_{i,j} \cup \mathcal{A}_{i,j}$ is a trivalent graph.
The edges of $\mathcal{K}_{i,j}$ correspond to the preimage of vertices of $\mathcal{A}_{i,j}$,
and hence have valence 3.   This establishes~(2) for $\wt{V}_{i,j}$.

There are 3 type of vertices: vertices of $\mathcal{K}_{i,j}$, 
intersection of edges of $\mathcal{K}_{i,j}$ with faces of $\wt V_{i,j}$,
and intersection of faces of $\mathcal{K}_{i,j}$ with edges of $\wt V_{i,j}$.
By construction, $\mathcal{K}_{i,j}$ has no vertices.  By Lemma~\ref{lem:HB}~(3)
the faces and edges of $\wt V_{i,j}$ are contained in its boundary.  Condition~(3)
follows.

Denote the genus of $\wt V_{i,j}$ by $n$.  Then 
$|\del \wt{P}_{i,j}| = n+1$. 
By Lemma~\ref{lem:complex}, $\mathcal{A}_{i,j}$ has at most $2n-1$ edges.
By construction, each edge of $\mathcal{A}_{i,j}$ corresponds to exactly one face
of $\mathcal{K}_{i,j}$.
Hence $\mathcal{K}_{i,j}$ has at most $2n-1$ faces.  By Lemma~\ref{lem:genus},
$n = g(\wt{V}_{i,j}) \leq C_{1} $; thus ${K}_{i,j}$ has at most $2C_1 - 1$ faces.  
 Similarly, every vertex of $\mathcal{A}_{i,j}$ corresponds to exactly one
 edge of $\mathcal{K}_{i,j}$ in the interior of $\wt V_{i,j}$.  Since the number of
 vertices of $\mathcal{A}_{i,j}$ is at most twice the number of its edges, we
 see that the number of edges of $\mathcal{K}_{i,j}$ in the interior
 of $\wt V_{i,j}$ is at most $4C_{1} - 2$.  This establishes~(4) for $\wt{V}_{i,j}$.

By construction, the components of $\widetilde{V}_{i,j}$ cut open along $\mathcal{K}_{i.,j}$
deformation retract onto $\wt{P}_{i,j}$ cut open along $\mathcal{A}_{i,j}$. 
It follows from Lemma~\ref{lem:complex}~(3) that $\mathcal{P}_{i,j}$ cut open along
$\mathcal{A}_{i,j}$ consists of disks,  and from Lemma~\ref{lem:complex}~(4) that
this is a single disk.  We conclude that $\wt V_{i,j}$ cut open along 
$\mathcal{K}_{i,j}$ is a single ball, establishing~(1)  for $\wt{V}_{i,j}$.
 
By Lemma~\ref{lem:HB}~(3) the projection of $\wt{V}_{i,j}$ to $V_{i,j}$ is a diffeomorphism.
Setting $K_{i,j}$ to be the image of $\mathcal{K}_{i,j}$ under the universal covering 
projection we obtain a complex
fulfilling the requirements of Lemma~\ref{lem:K}
 \end{proof}

\section{Proof of Proposition~\ref{pro:decomposition}}
\label{sec:ProofProp}
We use the notation of the previous sections.

We begin with the decomposition of $X$ given by $V_i \cap X = \{V_{i,j}\}_{j=1}^{n_{i}}$.

Fix one $V_{i,j}$ and consider its decomposition 
obtained by projecting the faces of $\wt V_{i,j}$ to $V_{i,j}$
(as discussed in Lemma~\ref{lem:voronoi}).  Recall that all the faces of this
decomposition are totally geodesic by construction.   
We decompose $V_{i,j}$ further using the faces of $K_{i,j}$, as described in Lemma~\ref{lem:K}.  
By Lemma~\ref{lem:K}~(2), these faces are totally geodesic as well.
By Lemma~\ref{lem:K}~(4), all the vertices of this decomposition are on $\del V_{i,j}$.
We first bound the number of these vertices:

\medskip

\noindent{\bf Claim.}  There is a universal $\bar{C}_{0}$ so that 
the number of vertices in $V_{i,j}$ is at most $\bar{C}_{0}$.

\medskip

\begin{proof}[Proof of claim]
By Lemma~\ref{lem:HB}~(3), the universal covering projection induces a diffeomorphism
between $V_{i,j}$ and $\wt V_{i,j}$.  It follows that a 
totally geodesic disk and a geodesic segment in $V_{i,j}$ intersect at most once; this
will be used below several times.
By Lemma~\ref{lem:K}~(2) all the the vertices are contained in $\del V_{i,j}$.

We first bound the number of vertices that lie in the interior of $X$.  
There are three types of vertices:

	\begin{description}
	\item[The intersection of three faces of $\del V_{i}$] By Lemma~\ref{lem:c0} there are
	at most $C_0$ such vertices.  (By transversality the intersection of more
	than three faces of $\del V_{i}$ does not occur.)
	\item[The intersection of an edge of $\del V_{i}$ with a face of $K_{i,j}$] 
	Since every face of $\del V_{i,j}$ is totally geodesic
	and every edge of $K_{i,j}$ is a geodesic segment, every face meets every edge at most
	once.  By Lemma~\ref{lem:c1} there are at most $C_1$ edges on $\del V_{i,j}$, and by 
	Lemma~\ref{lem:K} there are at most $2C_{1} - 1$ faces in $K_{i,j}$.  It follows
	that there are at most $C_{1}(2C_{1} - 1)$ vertices of this type.
	\item[The intersection of a face of $\del V_{i}$ and an edge of $K_{i,j}$]  
	Since every edge of $\del V_{i,j}$ is a geodesic segment
	and every face of $K_{i,j}$  is totally geodesic,
	every edge meets every face at most once.   By Lemma~\ref{lem:c2}
	there are at most $C_{2}$ edges on $\del V_{i,j}$.  It is clear that
	we are discussing only edges of $K_{i,j}$ that lie in the interior of $V_{i,j}$.
	By Lemma~\ref{lem:K}~(4) there are at most $4C_{1}-2$ such edges.  It follows that
	there are at most $C_{2}(4C_{1}-2)$ such vertices.
	\end{description}

Next we bound the number of vertices on $\del X$. There are two cases to consider.

	\begin{description}
	\item [An endpoint of an edge of $\del V_{i}$] Each such vertex is an endpoint of 
	a segment (as defined before Lemma~\ref{lem:c1}) and hence by that lemma there
	are at most $2C_1$ such vertices.
	\item [The intersection of a face of $K_{i,j}$ with $\del (V_{i} \cap \del X)$ and the 
	intersection of an edge of $K_{i,j}$ with $\del X$]  Every face of $K_{i,j}$ contributes
	at most two such vertices.  By Lemma~\ref{lem:K}~(4), there are at most $8C_{1} - 4$
	such points.
	\end{description}
The claim follows by setting (the different contributions are in brackets)
$$\bar{C} = [C_{0}] + [C_1(2C_1-1)] + [C_{2}(4C_{1}-2)] + [2C_1] + [8C_{1} - 4].$$
\end{proof}

Next, we subdivide each face into triangles by adding edges (note that this does not 
require faces to be disks).  This is done in $X$, so the
subdivision agrees on adjacent cells (including a cell that is adjacent to itself).  Note that the new
edges have valence 2.  Since the edges of $K_{i,j}$ have valence 3 and edges on the 
boundary have valence at most 3, all edges have valence at most 3.

By Lemma~\ref{lem:K}~(1) $V_{i,j}$ cut open along $K_{i,j}$ is a single ball.
Therefore there is a map from the closed ball $B$ onto $V_{i,j}$ that
is obtained by identifying disks on $\del B$ that correspond to the disks of $K_{i,j}$.  Since edges  
have valence at most 3, no
point of $V_{i,j}$ has more that 3 preimages.
The preimages of the triangulated faces induce a 
triangulation of $\del B$ with at most $3 \bar{C}_{0}$ vertices.  Denote the number of faces,
edges, and vertices in this triangulation by $f$, $v$, and $e$, respectively.  Note that
$3f = 2e$, or $e = \frac{3}{2}f$.  Euler characteristic gives: $2= f - e+v =   -\frac{1}{2} f + v$, or  
$f = 2v - 4$.  Thus, $f \leq 6 \bar{C}_{0} - 4$.  
We obtain a triangulation of $B$ by adding a vertex in the center of $B$, and coning every
vertex, edge, and triangle in $\del B$.  By construction there are exactly $f$ tetrahedra in this
triangulation.  The image of this triangulation gives a triangulation of $V_{i,j}$ that has 
at most $6\bar{C}_{0} - 4$ tetrahedra.  By Lemma~\ref{lem:HB}~(1), $n_{i} \leq C_{0}$.  
 Since $\{V_{i,j}\}_{{j=1}}^{n_{i}}$ are mutually
disjoint, by considering their union we obtain a triangulation of $V_{i} \cap X$
with at most $(6\bar{C}_{0} - 4)C_{0}$ tetrahedra.

By construction the triangulation of $V_{i} \cap X$ agrees with that of $V_{i'} \cap X$
on $(V_{i} \cap X) \cap (V_{i'} \cap X)$.  

Proposition~\ref{pro:decomposition} follows 
from this and Lemma~\ref{lem:c0} by setting 
$C = \max\{C_{3,}(6\bar{C}_{0} - 4)C_{0}\}$.

\providecommand{\bysame}{\leavevmode\hbox to3em{\hrulefill}\thinspace}
\providecommand{\MR}{\relax\ifhmode\unskip\space\fi MR }
\providecommand{\MRhref}[2]{%
  \href{http://www.ams.org/mathscinet-getitem?mr=#1}{#2}
}
\providecommand{\href}[2]{#2}



\begin{thebibliography}{1}

\bibitem{BenedettiPetronio}
Riccardo Benedetti and Carlo Petronio, \emph{Lectures on hyperbolic geometry},
  Universitext, Springer-Verlag, Berlin, 1992. \MR{MR1219310 (94e:57015)}
 
 
\bibitem{breslin}  
Breslin, William,\emph{Thick triangulations of hyperbolic $n$-manifolds},
Pacific J. Math. \textbf{241} (2009), no.~2, 215--225, \MR{2507575 (2010b:30066)}

\bibitem{decay}
Bennett Chow, Sun-Chin Chu, David Glickenstein, Christine Guenther, James
  Isenberg, Tom Ivey, Dan Knopf, Peng Lu, Feng Luo, and Lei Ni, \emph{The
  {R}icci flow: techniques and applications. {P}art {I}}, Mathematical Surveys
  and Monographs, vol. 135, American Mathematical Society, Providence, RI,
  2007, Geometric aspects. \MR{MR2302600 (2008f:53088)}

\bibitem{hatcher}
Allen Hatcher, \emph{Algebraic topology}, Cambridge University Press,
  Cambridge, 2002. \MR{1867354 (2002k:55001)}

\bibitem{kojima}
Sadayoshi Kojima, \emph{Isometry transformations of hyperbolic
  {$3$}-manifolds}, Topology Appl. \textbf{29} (1988), no.~3, 297--307.
  \MR{MR953960 (90c:57033)}

\bibitem{thurston}
William~P Thurston, \emph{{The Geometry and Topology of Three-Manifolds}},
  http://www.msri.org/publications/books/gt3m/, 1977.

\end{thebibliography}

\end{document}